\documentclass[12pt,reqno]{amsart}

\usepackage{amsmath}
\usepackage{amsthm}
\usepackage{hyperref}
\usepackage[margin=1.0in]{geometry}

\numberwithin{equation}{section}

\newtheorem{prop}{Proposition}
\newtheorem{lemma}[prop]{Lemma}
\newtheorem{thm}[prop]{Theorem}
\newtheorem{theorem}[prop]{Theorem}

\numberwithin{prop}{section}

\theoremstyle{definition}
\newtheorem{defn}[prop]{Definition}
\newtheorem{ex}[prop]{Example}

\newtheorem{rmk}[prop]{Remark}
\newtheorem{remark}[prop]{Remark}

\newcommand{\dt}{\frac{\partial}{\partial t}}
\newcommand{\brs}[1]{\left| #1 \right|}
\newcommand{\gG}{\Gamma}
\newcommand{\gD}{\Delta}
\newcommand{\gd}{\delta}

\newcommand{\gw}{\omega}
\newcommand{\ga}{\alpha}
\newcommand{\gb}{\beta}
\renewcommand{\gg}{\gamma}

\renewcommand{\ge}{\epsilon}
\newcommand{\N}{\nabla}
\renewcommand{\bar}[1]{\overline{#1}}
\newcommand{\del}{\partial}

\newcommand{\HH}{\mathcal H}
\newcommand{\YY}{\mathcal Y}
\newcommand{\ZZ}{\mathcal Z}
\newcommand{\FF}{\mathcal F}

\newcommand{\IP}[1]{\langle #1\rangle}
\newcommand{\CB}[1]{[ #1 ]}
\newcommand{\G}{\mathbf{G}}
\newcommand{\Ric}{\mathrm{Ric}} 
\newcommand{\R}{\mathbf{R}}

\newcommand{\til}[1]{\widetilde{#1}}

\DeclareMathOperator{\Sym}{Sym}
\DeclareMathOperator{\Rc}{Rc}

\DeclareMathOperator{\Id}{Id}
\DeclareMathOperator{\id}{id}

\DeclareMathOperator{\End}{End}

\newcommand{\hook}{\mathbin{\hbox{\vrule height2.4pt width4.5pt depth-2pt
\vrule height5pt width0.4pt depth-2pt}}}

\newcommand{\dth}{e_{\theta}}

\newcommand{\nc}{\newcommand}\nc{\br}{\overline}
\nc{\FH}{\mathcal H} \nc{\CC}{\mathbb C}\nc{\DD}{\mathcal D}\nc{\BB}{\mathbb B}
\nc{\Cc}{\mathcal
    C}\nc{\JJ}{\mathbb J}\nc{\II}{\mathbb I}\nc{\I}{\mathcal{I}}
\nc{\PP}{\mathbb P}
\nc {\OO}{\mathcal{O}} \nc{\lra}{\longrightarrow}
\nc{\bdot}{\bullet} \nc{\w}{\omega}
\nc{\dd}{\mathcal{D}}\nc{\im}{\mathrm{Im}}\nc{\Nij}{\mathrm{Nij}}\nc{\Jac}{
\mathrm{Jac}}
 \nc{\sgn}{\mathrm{sgn}} 
\newcommand{\Hom}{\mathrm{Hom}} 

\newcommand{\so}{\mathfrak{so}}\nc{\MM}{\mathcal{M}}\nc{\Har}{\mathcal{H}}
\nc{\zed}{\mathcal{Z}}
\nc{\bb}[1]{\mathbb{#1}} \nc{\image}{\mathrm{Im}\ } 
 \nc{\ba}{\overline}
\nc{\AAA}{\mathcal{A}} \nc{\de}{\delta}\nc{\debar}{\overline{\delta}}
\nc{\GG}{\mathbb{G}} \nc{\Ja}{e^{\tfrac{\pi}{2}\JJ_1}}
\nc{\Jb}{e^{\tfrac{\pi}{2}\JJ_2}}
\nc{\eps}{\epsilon}\nc{\Dir}{\mathrm{Dir}}\nc{\SO}{\mathrm{SO}}
\nc{\IPS}[1]{#1}\nc{\veps}{\varepsilon}\nc{\Cour}[1]{[#1]}
\nc{\ad}{\mathrm{ad}} \nc{\UU}{\mathcal{U}} \nc{\Aa}{\mathcal
    A}\nc{\Bb}{\mathcal{B}}\nc{\Pic}{\mathrm{Pic}}\nc{\type}{\mathrm{type}}
\nc{\TT}{\mathbb{T}}\nc{\T}{\mathcal{T}} \nc{\cl}{\mathrm{cl}}\nc{\isom}{\cong}
\nc{\Z}{\mathcal{Z}}\nc{\Tot}{\mathrm{Tot}}\nc{\Dol}{\mathrm{Dol}}\nc{\hol}{
\mathrm{hol}}
\nc{\EE}{\mathbb{E}}\nc{\e}{\mathbf{e}} \nc{\E}{\mathbf{E}} 
\nc{\Gg}{\mathfrak{g}}
\nc{\Der}{\mathrm{Der}}\nc{\ol}{\overline}

\nc{\RIC}{\mathbf{Ric}}
\nc{\B}{\mathbf{B}}
\nc{\V}{\mathbf{V}}
\nc{\J}{\mathbf{J}}
\nc{\Q}{\mathbf{Q}}
\nc{\A}{\mathbf{A}}
\nc{\W}{\mathbf{W}}

\begin{document}

\title[Generalized geometry, T-duality, and renormalization group
flow]{Generalized geometry, T-duality, and renormalization group flow}

\date{October 18, 2013}

\author{Jeffrey Streets}
\address{Rowland Hall\\
         University of California, Irvine\\
         Irvine, CA 92617}
\email{\href{mailto:jstreets@uci.edu}{jstreets@uci.edu}}

\begin{abstract} We interpret the physical $B$-field renormalization group flow
in the language of Courant
algebroids, clarifying the sense in which this flow is the natural ``Ricci
flow'' for generalized geometry.  Next we show that the $B$-field
renormalization group flow preserves
T-duality in a natural sense.  As corollaries we obtain new long time existence
results for the $B$-field renormalization group flow.
\end{abstract}

\thanks{The author gratefully acknowledges support from the National Science
Foundation and the Sloan Foundation.}

\maketitle

\section{Introduction}

Let $(M^n, g)$ be a Riemannian manifold and let $H_0 \in \Lambda^3(T^* M)$,
$dH_0 =
0$.  Given this setup and $b \in \Lambda^2(M)$ we set $H = H_0 + db$.  The
\emph{B-field renormalization group flow} is the system of equations
\begin{gather} \label{RGflow}
\begin{split}
\dt g_{ij} =&\ - 2 \Rc_{ij} + \frac{1}{2} H_{i p q} H_{j}^{\ pq},\\
\dt b =&\ - d^*_g H.
\end{split}
\end{gather}
The physical interpretation of $H$ is, in analogy with Yang-Mills
theory, as a generalized magnetic field strength. With background fields $g$ and
$H$ one
can define an energy for string worldsheets in this target geometry, and
equation (\ref{RGflow}) arises by imposing cutoff independence for the
associated quantum field theory at one-loop.  These ideas began in the work of
Friedan et. al. (\cite{Friedan2}, \cite{Friedan}, \cite{Friedan3}, \cite{FMS}). 
For the sequel we require a gauge-fixed version of this flow.  In particular,
given the above setup and a one-parameter family of functions $f_t$, consider
\begin{gather} \label{GFRGflow}
\begin{split}
\dt g_{ij} =&\ - 2 \Rc_{ij} + \frac{1}{2} H_{i p q} H_{j}^{\ pq} + (L_{\N f}
g)_{ij},\\
\dt b =&\ - d^*_g H + i_{\N f} \hook H.
\end{split}
\end{gather}

The first purpose of this paper is to give equation (\ref{RGflow}) a natural
interpretation in terms of generalized geometry.  This subject was initiated in
the work of Hitchin \cite{Hitchin}, 
and later developed in the thesis of Gualtieri \cite{GualtThesis}. 
Partly inspired by
physical ideas, generalized geometry treats not just the tangent bundle,
 but a twisted Courant algebroid $\E$ modeled on $T \oplus T^*$, as the
fundamental object associated to a smooth manifold.  With this philosophy 
one is lead to the definition of a generalized metric $\G$ (cf. \S
\ref{metric}), which naturally incorporates a standard Riemannian metric $g$ and
a
 two-form $b$.  Likewise, associated to $\G$ there are two canonical connections
on $\E$ referred to as Bismut connections.  These connections in 
turn have natural ``Ricci tensors" $\R$ interpreted as elements of $\so(E)$ (see
\S \ref{Bismut} for the precise definitions).  In analogy with the usual Ricci
flow equation, one is lead (\cite{GualtPr}) to define the \emph{generalized
Ricci flow} equation
\begin{gather} \label{GRflow}
\dt \G = -2 [\R, \G].
\end{gather}
\noindent Moreover, one can hope that this geometrically motivated construction involving data $(g,b)$
leads to
the same flow equations (\ref{RGflow}) derived from physical considerations. 
Gualtieri has shown that this is the case, and we present a proof of this fact
here.

\begin{thm} \emph{(\cite{GualtPr})} \label{GflowBflowequiv} Let $(E, \pi,
[\cdot,
\cdot])$ be an exact
real Courant algebroid over a smooth manifold $M$.  A one-parameter 
family of generalized metrics $\G_t$ is a solution of generalized Ricci flow if
and only if the one parameter family of induced pairs $(g_t,b_t)$ of metrics and
torsion potentials solve the $B$-field renormalization group flow.
\end{thm}

The second main purpose of this paper is to exhibit the relationship between
equation (\ref{GFRGflow}) and T-duality.  T-duality is an equivalence between
different quantum field theories which have very different classical
descriptions.  This phenomena was first discovered in 1987 by
Buscher \cite{Buscher}, \cite{B2}, and was further explored by Ro\u{c}ek and
Verlinde in
\cite{RV}.  More recently Cavalcanti and Gualtieri \cite{CG} gave a
unified
description of the T-duality relationship for all structures in generalized
geometry, and in particular for generalized K\"ahler structure.  Due to the
role the renormalization group flow (\ref{RGflow}) plays in the theory of
nonlinear sigma models, physically motivated arguments suggest that T-duality of
pairs of metrics and flux potentials $(g, b)$ should be preserved.  In
particular, Haagensen \cite{Haag} addresses this question using some
explicit coordinate calculations and some physical arguments.  The next
theorem gives a completely rigorous treatment of this idea from a purely
geometric point of view, which moreover makes clear the change in dilaton which
occurs.

\begin{thm} \label{Tdualflow} Suppose $(M^n, H, \theta)$ and $(\bar{M}, \bar{H},
\bar{\theta})$ are topologically T-dual circle bundles (cf.
Definition \ref{toptdual}).  Given $(g, b)$ an $S^1$-invariant pair of metric and two-form, and $f_t$ a one-parameter family of
$S^1$-invariant functions, let $(g_t, b_t)$ be the unique
solution to (\ref{GFRGflow}) with this initial condition.  Let $(\bar{g}_t,
\bar{b}_t)$ 
denote the one-parameter family of T-dual pairs to $(g_t, b_t)$.  Then
$(\bar{g}_t, \bar{b}_t)$ is the unique solution to (\ref{GFRGflow}) with initial
condition $(\bar{g}, \bar{b})$ with $\bar{f}_t = f_t + \log \phi_t$,
where $\phi_t =
g_t(\dth, \dth)$ is the function determining the length of the circle fiber on
$M$ at each time $t$.
\end{thm}

\begin{rmk} This theorem holds for arbitrary T-dual torus bundles by taking
repeated application of $S^1$ dualities, so for
simplicity we give the proof in the case of circle bundles.
\end{rmk}

While Theorem \ref{Tdualflow} completely captures the relationship of
$T$-duality to the renormalization group flow of general pairs $(g, b)$, a
number of questions still remain.  For instance, through the work of the author
and Tian (\cite{ST1}, \cite{ST2}, \cite{ST3}) it was discovered
that, after coupling to appropriate evolution equations for the complex
structures, equation (\ref{RGflow}) preserves generalized K\"ahler geometry. 
One can ask whether the $T$-duality relationship for these complex structures is
preserved, which certainly seems likely.  Moreover, we remark here that Theorem
\ref{Tdualflow} may play a role in the
singularity analysis of equation (\ref{RGflow}).  For instance, rescaling limits
of solutions to (\ref{RGflow}) may either converge or collapse depending on an
appropriate injectivity radius estimate.  In the collapsing case these solutions
inherit the geometry of an invariant metric on a principal torus bundle.  Thus
Theorem \ref{Tdualflow} provides a ``dual'' model for such singularities.

Here is an outline of the rest of the paper.  In \S \ref{backgrnd} we recall
some background on the fundamental constructions of generalized geometry.  
In \S \ref{GRFsection} we compute variational equations for generalized metrics
and prove Theorem \ref{GflowBflowequiv}.  Next in \S \ref{Tdualtop} 
we recall some results related to topological T-duality and in \S
\ref{Tdualgeom} we recall T-duality transformations of geometric structures. 
Lastly in 
\S \ref{Tdualflowsec} we prove Theorem \ref{Tdualflow} and give a number of
examples illustrating the theorem.

\textbf{Acknowledgements:} The author thanks Mark
Stern for introducing him to equation (\ref{RGflow}) and for interesting and
helpful conversations on $T$-duality.  Furthermore, the author thanks Marco
Gualtieri, who played a significant role in the development of this work by
informing the author of Theorem 
\ref{GflowBflowequiv}, and answering many questions related to T-duality.

\section{Background on generalized geometry} \label{backgrnd} 
\subsection{Courant algebroids} \label{courant}

Let $\E$ be an exact Courant algebroid, with extension
\[
0\lra T^* \overset{\pi^*}{\longrightarrow}  \E \overset{\pi}{\lra} T \lra 0
\]
and with neutral metric $\IP{\cdot,\cdot}$.  Throughout, we identify $\E$ with
$\E^*$, using this
metric. The Courant bracket is $\CB{\cdot,\cdot}$, and upon choosing an
isotropic splitting
$s:T\to\E$ of $\pi$, we obtain a closed 3-form
\begin{equation}\label{3form}
    H_s(X,Y,Z) =\IP{\CB{sX,sY},sZ}.
\end{equation}
Any 2-form $b\in\Omega^2(M,\mathbb R)$ defines a Lie algebra element
$b^\pi:\E\to\E$
via
\begin{equation}\label{bfield}
    b^\pi = \pi^*b \pi. 
\end{equation}
Exponentiating, we obtain the orthogonal map
\[
e^{b} = \id_\E + b^\pi \in SO(\E),
\]
which satisfies $\pi e^b = \pi$.  Therefore, given a splitting $s$ of $\pi$, it
follows that $e^b s$
is a new splitting, and then one computes that
\begin{align} \label{Hchange}
H_{e^b s} = H_s + db.
\end{align}

\subsection{Generalized metrics} \label{metric}

\begin{defn} Given $\E$ an exact Courant algebroid, a \emph{generalized metric}
is an endomorphism $\G : E \to E$ satisfying
\begin{enumerate}
 \item $\G^2 = \Id$
 \item $\G^* = \G$
 \item $\IP{\G \cdot, \cdot}$ is positive definite.
\end{enumerate}
\end{defn}

This definition can also be expressed in terms of subbundles of $\E$.  In
particular, note that 
the choice of a maximal positive-definite subbundle $V_+\subset \E$ defines a
reduction in structure group from $O(n,n)$ to $O(n)\times O(n)$. The
orthogonal complement
$V_-\subset \E$ of the bundle $V_+$ is then negative-definite, and we obtain a
direct sum decomposition
\[
\E = V_+\oplus V_-.
\]
From this point of view, if we let $P_\pm$ denote the neutral orthogonal
projections to $V_\pm$,
we recover the metric by
\[
\G = P_+ - P_-.
\]
Conversely, given $\G$ a generalized metric, we obtain projection operators
\[
P_\pm = \tfrac{1}{2}(\Id \pm \G).
\]

A generalized metric induces various classical objects.  First, we obtain a
usual Riemannian metric, thought of as a map $g^{-1} :T^* \to T$, via
\[
g^{-1} = \pi \G \pi^*.
\]
It also defines a splitting $s_{\G}:T\to\E$ of the Courant algebroid, given
by
\begin{equation}\label{splitg}
    s_\G = \G\pi^*g.
\end{equation}
Notice that $\pi s_\G = \pi \G \pi^* g = \id_T$.  With a choice of background
splitting $s_0$, this defines a torsion potential $b$ via
\begin{align*}
 s_{\G} = e^b s_0
\end{align*}
This in turn defines a closed three-form via (\ref{Hchange}), denoted $H_G$.

Conversely, we can use a metric $g$ and an isotropic splitting $s$ to induce a
generalized metric on $\E$.  Specifically, observe the consequences of the above
equations,
\begin{align} \label{gstog}
 \G \pi^* = s_{\G} g^{-1}, \qquad \G s_{\G} = \pi^* g.
\end{align}
Given $s$ and $g$, these equations can be taken as the definition of the
endomorphism $\G$ on the image of $\pi^*$ and $s$, which suffices to define
$\G$.

\subsection{Lie algebra}\label{lieal}

Consider the decomposition $\E = V_+\oplus V_-$ defined
by $\G$ as above.  We can decompose $R \in \so(\E)$ as
\[
R = (P_++P_-) R (P_+ + P_-) = R_+ + R_- + S_+ + S_-,
\]
where $R_\pm=P_\pm R P_\pm\in\so(V_\pm)$  and $S_\pm=P_\mp R P_\pm : V_\pm\to
V_\mp$, and the latter are equivalent data via
\begin{equation*}
S_\pm^* = (P_\mp R P_\pm)^* = - P_\pm R P_\mp = - S_\mp,
\end{equation*}
where we implicitly identify $V_+ = V_+^*$ and $V_- = V_-^*$ using the neutral
metric.  Since all
generalized metrics are related by a neutral orthogonal transformation, it
is natural to interpret an infinitesimal change in $\G$ is given by an element
of the form
$S_+$.  

Now fix $h \in \Sym^2 T^*$ and $k \in \Lambda^2 T^*$.  We interpret $h+k$ as a
map $T\to T^*$, and then we can define a
transformation $\eta = \pi^*(h+k)\pi:\E\to\E$.  Observe that $\eta \notin
\so(\E)$,
because
$\eta^*=\pi^*(h-k)\pi$ does not coincide with $-\eta = \pi^*(-h-k)\pi$ unless
$h=0$.  Nevertheless, $\eta$ does
determine a Lie algebra element by setting 
\[
S_+ = P_-\eta P_+.
\]
This forces $S_- = -S_+^* = - P_+\eta^* P_-$, and we obtain a Lie algebra
element 
\[
R = P_-\eta P_+ - P_+\eta^* P_-.
\]
This discussion is summarized in the following proposition.
\begin{prop}\label{infsym}
Let $h\in \Sym^2T^*, k\in \wedge^2 T^*$.  These data determine a Lie algebra
element
\begin{equation}\label{lieform}
R_{h,k} = P_- \eta P_+ - P_+\eta^* P_- \in \so(\E),
\end{equation}
where $\eta = \pi^*(h+k)\pi$.
\end{prop}

\subsection{Bismut Connections}\label{Bismut}
A generalized metric naturally determines two connections on $T$, called
\emph{Bismut connections}.  In particular, let $s_{\pm} = (\G \pm \Id) \pi^* g$, let $X^{\pm} := s_{\pm} X$, 
and consider
\begin{gather} \label{Bismutconn}
\nabla^\pm_X Y = \pi P_\pm\CB{X^\mp, Y^\pm}.
\end{gather}
These connections have torsion $T^\pm$ such that
\[
g(T^\pm(X,Y),Z) = \pm H(X,Y,Z),
\]
and if $\N$ denotes the Levi-Civita connection, we have
\[
\nabla^\pm = \N \pm \tfrac{1}{2} g^{-1}H,
\]
where $g^{-1}H$ denotes the composition
\begin{align*}
T \overset{H}{\lra} \Hom(T, T^*) \overset{g^{-1}}{\lra} \End(T).
\end{align*}
Since the connections $\N^{\pm}$ have torsion, the Ricci tensor is no longer
symmetric.  A direct calculation yields
\begin{lemma} \label{generalBismutRicci} Let $(M^n, g)$ be a Riemannian manifold
and let $\N^{\pm} = \N \pm \tfrac{1}{2} g^{-1} H$ as above.  Then
\begin{align} \label{Bismutricci}
\Ric(\nabla^\pm) = \Ric(\nabla) - \tfrac{1}{4} H^2 \mp \tfrac{1}{2}d^*H.
\end{align}
where
\begin{align*}
H^2(X, Y) =&\ \left< i_X H, i_Y H \right>.
\end{align*}
\end{lemma}

We can use the symmetric and skew symmetric pieces of this Ricci tensor to
define a Lie algebra element in accordance with Proposition \ref{infsym}.

\begin{defn} \label{genric} Given $\E$ an exact Courant algebroid and $\G$ a
generalized metric, the \emph{generalized Ricci tensor}, $\R$, is the Lie
algebra element associated to $\Rc(\N^-)$ via Proposition \ref{infsym}. 
\end{defn}

\section{Generalized Ricci flow} \label{GRFsection}
\subsection{Variational Formulas}

In this subsection we compute variation formulas for one-parameter families of
generalized metrics and the various associated data.

\begin{defn} Let $(\E, \pi, q, [\cdot, \cdot])$ be an
exact real Courant algebroid over a smooth manifold $M$.  Given $h + k
\in T^* \otimes T^*$, let $\V$ denote the Lie algebra element associated to
$h+k$ via Proposition \ref{infsym}.
We say a one-parameter family of sections $\A_t \in \End(\E)$ has
\emph{variation $\V$} if
\begin{align*}
\left.\dt \A \right|_{t = 0} = [\V,\A].
\end{align*}
\end{defn}

\begin{rmk} We observe that, by using the section $s_{\G}$ to provide an
isomorphism $E \cong T \oplus T^*$, $\V$ can be written in matrix form as
\begin{align} \label{Vmatrix}
\V = \tfrac{1}{2}\begin{pmatrix} -g^{-1}h &
-g^{-1}kg^{-1}\\k&hg^{-1}\end{pmatrix}.
\end{align}
This form of $\V$ will make the calculations to follow more transparent.
\end{rmk}

\begin{lemma} \label{metricvariation} Let $(\E, \pi, q, [\cdot, \cdot])$ be an
exact real Courant algebroid over a smooth manifold $M$.  Suppose $\G_t$ is a
one-parameter family of generalized metrics with variation $\V$ and $\G_0 = \G$.
Let $g_t, s_{\G_t}$ denote the associated Riemannian metrics and splittings as
above.  Moreover, assume a background section $s_0$ and define a family of torsion
potentials $b_t$ via $s_{\G_t} = e^{b_t} s_0$.  Then
\begin{align*}
\left. \tfrac{\del}{\del t} g \right|_{t = 0} =&\ h,\\
\left. \tfrac{\del}{\del t} s_{\G} \right|_{t = 0} =&\ \pi^* k\\
\left. \tfrac{\del}{\del t} b \right|_{t = 0} =&\ k,\\
\end{align*}
\begin{proof}
First note that
\begin{align*}
[\V, \G] = \begin{pmatrix}
-g^{-1}k & -g^{-1}hg^{-1}\\
h & kg^{-1}
\end{pmatrix}.
\end{align*}
We derive the evolution equation for $g$.  Differentiating the defining relation
$g^{-1} = \pi \G \pi^*$ yields
\begin{align*}
- g^{-1} \tfrac{\del}{\del t} g g^{-1} =&\ \tfrac{\del}{\del t} g^{-1}\\
=&\ \pi [\V, \G] \pi^*\\
=&\ \pi \begin{pmatrix}
-g^{-1}k & -g^{-1}hg^{-1}\\
h & kg^{-1}
\end{pmatrix} \pi^*\\
=&\ - g^{-1} h g^{-1}.
\end{align*}
The first claim follows.  Next we differentiate the equation $s_{\G} = \G \pi^*
g$ to obtain
\begin{align*}
\tfrac{\del}{\del t} s_{\G} =&\ [\V, \G] \pi^* g + \G \pi^* h = \begin{pmatrix}
-g^{-1}k & -g^{-1}hg^{-1}\\
h & kg^{-1}
\end{pmatrix}
\begin{pmatrix}
0 & 0\\
g & 0
\end{pmatrix}
+ \begin{pmatrix}
0 & g^{-1}\\
g & 0
\end{pmatrix}
\begin{pmatrix}
0 & 0\\
h & 0
\end{pmatrix} = \begin{pmatrix}
0 & 0\\
k & 0
\end{pmatrix}.
\end{align*}
The second claim follows.  Noting the equation $s_{\G} = e^b s_0 = s_0 + \pi^* b
\pi s_0$, one derives $\pi^* \tfrac{\del}{\del t} b = \pi^* k$, and so
$\tfrac{\del}{\del t} b = k$.  The last claim follows.
\end{proof}
\end{lemma}

Having determined how a variation in generalized metric induces variations of
other relevant quantities we now go backwards and determine the evolution
equation for $\G$ induced by an evolution equation for $g$ and $s$.

\begin{defn} Let $(E, \pi, [\cdot, \cdot])$ be an exact real Courant algebroid
over $M$.  We say that a one parameter family $(g_t, s_t)$ of metrics on $M$ and
isotropic splittings of $E$ has \emph{variation $(h, k)$} if
\begin{align*}
\left. \dt g \right|_{t = 0} =&\ h.\\
\left. \dt s \right|_{t = 0} =&\ \pi^* k,
\end{align*}
where $h \in \Sym^2(T^*M)$ and $k \in \Lambda^2 T^*$.
\end{defn}

\begin{lemma} \label{reversevariation} Let $(E, \pi, [\cdot, \cdot])$ be an
exact real Courant algebroid over $M$.  Fix $(g_t, s_t)$ a one-parameter
family of metrics on $M$ and splittings with variation $(h, k)$, and let $\G_t$
denote the one-parameter family of generalized metrics associated to this data
via (\ref{gstog}).  Then
\begin{align*}
\left. \dt \G \right|_{t = 0} =&\ [\V, \G]
\end{align*}
where $V$ is the Lie algebra element associated to $h + k$ as in
(\ref{Vmatrix}).
\begin{proof} Express the splitting $s_t = e^{b_t} s_0 = e^{b_t} s_{\G_0}$, and
then $b_0 = 0$.  Then for general $t$ we have the equation
\begin{align*}
\G_{g, b} =&\ \begin{pmatrix} - g^{-1} b & g^{-1} \\ g - b g^{-1} b & b g^{-1}
\end{pmatrix}.
\end{align*}
We can directly differentiate, using $b_0 = 0$ to yield the result.
\end{proof}
\end{lemma}

\subsection{Generalized Ricci flow}

\begin{defn} Let $(E, \pi, [\cdot, \cdot])$ be an exact real Courant algebroid
over a smooth manifold $M$.
 We say that a one-parameter family of generalized metrics $\G_t$ is a solution
of \emph{generalized Ricci flow} if
\begin{align} \label{GRF}
\dt \G = [-2\R, \G],
\end{align}
where $\R$ is the generalized Ricci curvature of $\G$ as in Definition
\ref{genric}
\end{defn}

\begin{proof}[Proof of Theorem \ref{GflowBflowequiv}] It follows directly from
Lemma  \ref{metricvariation} that the
evolution equations for the induced pair $(g_t,b_t)$ are precisely those of
(\ref{RGflow}),
as required.  Conversely, suppose $(g_t, b_t)$ are a solution to (\ref{RGflow}).
 It follows from Lemma \ref{reversevariation} that the associated generalized
metrics $\G_t$ evolve by
\begin{align*}
\dt \G = \CB{\V, \G},
\end{align*}
where $\V$ is the Lie algebra element associated to $\dt(g + b)$.  Comparing
with Lemma \ref{generalBismutRicci} again we see that $\V_t = -2 \R_t$ as
defined above, and so the theorem follows.
\end{proof}

\section{Topological T-duality} \label{Tdualtop}

In this section we recall some background on the topological aspect of
T-duality.  Our discussion here follows closely the work of Cavalcanti-Gualtieri
\cite{CG}.

\begin{defn} Let $M$, $\bar{M}$ be principal $T^k$ bundles over a common base
manifold $B$, and let $H \in \Omega^3_{T^k}(M)$ and 
$\bar{H} \in \Omega^3_{T^k}(\bar{M})$ be invariant closed forms, and finally let
$\theta$ and $\bar{\theta}$ denote connection $1$-forms
 on $M$ and $\bar{M}$.  Consider $M \times_B \bar{M}$ the fiber product of $M$
and $\bar{M}$, with projection maps
 $p : M \times_B \bar{M} \to M, \bar{p} : M \times_B \bar{M} \to \bar{M}$.  We
say that $(M,H,\theta)$ and $(\bar{M},\bar{H},\bar{\theta})$ are
\emph{topologically $T$-dual} if
\begin{align} \label{toptdual}
p^* H - \bar{p}^* \bar{H}= d(p^* \theta \wedge \bar{p}^* \bar{\theta}).
\end{align}
\end{defn}

\begin{remark} \label{tdualrmk} While as written this definition requires
specific choices of $H$ and $\bar{H}$, the definition only depends on the
cohomology
 classes $[H]$ and $[\bar{H}]$.  Specifically, if $(M, H,\theta)$ and $(\bar{M},
\bar{H},\bar{\theta})$ are $T$-dual, and we set $H' = H + db$, with $b \in
\Omega^2_{T^k}(M)$,
 there exists a new connection $\theta'$ on $M$ and also $\bar{H}',
\bar{\theta}'$ on $\bar{M}$ such that for the quadruple $(H',\theta',\bar{H}',
\bar{\theta}')$
 the relation (\ref{toptdual}) holds.  In particular, as a corollary of Lemma
\ref{consistencylemma} we may choose any $S^1$-invariant metric $g$ whose
induced connection $1$-form is $\theta$ and then take the T-dual data to $(g,b)$
provides the requisite data.
\end{remark}

\begin{theorem} (\cite{BEM} Theorem 3.1)  If $(M, H)$ and $(\bar{M}, \bar{H})$
are $T$-dual with $p^* H - \bar{p}^* \bar{H} = dF$, then
\begin{gather} \label{taudefn}
\begin{split}
\tau : (\Omega_{T^k}(M), d_H) \to (\Omega_{T^k}(\bar{M}), d_{\bar{H}}), \quad
\tau(\rho) = \int_{T^k} e^F \wedge \rho
\end{split}
\end{gather}
is an isomorphism of differential complexes, where the integration is along the
fibers of $M \times_B \bar{M} \to \bar{M}$.
\end{theorem}

\begin{remark} The map $\tau$ is a map on the Clifford module of $T^k$-invariant
forms.  To show that it is an isomorphism of Clifford modules 
we require an isomorphism $\phi : (TM \oplus T^*M) / T^k \to (T \bar{M} \oplus
T^*\bar{M}) / T^k$, which we define next.
\end{remark}

\begin{defn} \label{phidef} Given $(X + \xi) \in (TM \oplus T^*M) / T^k$, choose
the unique lift $\hat{X}$ of $X$ to $T(M \times \bar{M})$ such that
\begin{align*}
p^* \xi(Y) - F(\hat{X}, Y) = 0, \quad \mbox{ for all } Y \in \mathfrak t^k_{{M}}
\end{align*}
Due to this condition the form $p^* \xi - F(\hat{X}, \cdot)$ is basic for the
bundle determined by $\bar{p}$, and can therefore be pushed forward to
$\bar{M}$.  We define a map
\begin{align*}
\phi(X + \xi) = \bar{p}_*(\hat{X}) + p^* \xi - F(\hat{X}, \cdot).
\end{align*}
\end{defn}

\begin{lemma} \label{phicohomcond} The map $\phi$ defined above depends only on
$[H]$ and $[\bar{H}]$.
\begin{proof} Following the discussion in Remark \ref{tdualrmk}, if $H' = H +
dB$ then
\begin{align*}
p^* H' - \bar{p}^* \bar{H} = d(F + p^* B) =: dF'
\end{align*}
Moreover, the action of $p^* B$ on $\mathfrak t_M^k \otimes \mathfrak
t_{\bar{M}}^k$ is trivial.  Hence when lifting vectors to the configuration
space as in 
Definition \ref{phidef}, using either $F$ or $F'$ yields the same result, and so
the lemma follows.
\end{proof}
\end{lemma}

\section{Geometric T-duality} \label{Tdualgeom}

In this section we present the notion of $T$-duality for generalized metrics. 
We
take as background data topologically T-dual $S^1$-bundles $(M, H, \theta)$ and
$(\bar{M}, \bar{H}, \bar{\theta})$.  The metric
data then consists of an $S^1$-invariant metric $g$ on $M$ and an
$S^1$-invariant two-form $b$ on $M$.  In \cite{Buscher}, \cite{B2}
Buscher discovered a way to transform this data, as well as an auxiliary
dilaton, to the manifold $\bar{M}$ in such a way that fixed points of
(\ref{RGflow}) on $M$ are transformed into fixed points of (\ref{GFRGflow}) with
a particular choice of $f_t$ on $\bar{M}_t$.  
The content of Theorem \ref{Tdualflow} is to show that this behavior persists
for general solutions of (\ref{RGflow}).

\subsection{Duality of geometric structures}

\begin{defn} \label{dualmetric} Let $(M, H,\theta)$ and $(\bar{M},
\bar{H},\bar{\theta})$ be T-dual.  Given $\G$ a generalized metric on $(TM
\oplus T^*M) / T^k$, the \emph{dual metric} is
\begin{gather} \label{Tdualmetrics}
\bar{\G} := \phi \G \phi^{-1}
\end{gather}
\end{defn}

\begin{remark} The simplicity of this definition illustrates the value of
adopting the viewpoint of Courant algebroids.  Indeed, using the map $\phi$ it
is possible to easily define T-duality
 transformations for other natural objects such as generalized complex
structures.  By working out the induced map on $(g,b)$ one recovers the famous
``Buscher rules,'' \cite{Buscher}, \cite{B2}, which we now record.
\end{remark}

Given $(M, H,\theta)$ and $(\bar{M}, \bar{H},\bar{\theta})$
T-dual bundles with connections $\theta$ and $\bar{\theta}$, recall that an
$S^1$-invariant generalized metric $\G$ is determined by an $S^1$ invariant pair
$(g, b)$ of metric and two-form potential on $M$, which can be expressed as
\begin{gather} \label{generalmetric}
 \begin{split}
 g =&\ g_0 \theta \odot \theta + g_1 \odot \theta + g_2\\
 b =&\ b_1 \wedge \theta + b_2  
 \end{split}
\end{gather}
where $g_i$ and $b_i$ are basic forms of degree $i$.  

\begin{lemma} (\textbf{Buscher Rules}) Suppose $(M, H, \theta)$ and $(\bar{M},
\bar{H}, \bar{\theta})$ are topologically T-dual.  Given $\G$ an
 $S^1$-invariant generalized metric on $TM \oplus T^*M$ and $\bar{\G} = \phi \G
\phi^{-1}$ the dual metric on $T \bar{M} \oplus T^* \bar{M}$,
 if the pair $(g,b)$ associated to $\G$ is given by (\ref{generalmetric}), then
the pair $(\bar{g}, \bar{b})$ determined by $\G$ takes the form
\begin{gather} \label{bus}
 \begin{split}
  \bar{g} =&\ \frac{1}{g_0} \bar{\theta} \odot \bar{\theta} - \frac{b_1}{g_0}
\odot \bar{\theta} + g_2 + \frac{b_1 \odot b_1 - g_1 \odot g_1}{g_0}\\
 \bar{b} =&\ - \frac{g_1}{g_0} \wedge \bar{\theta} + b_2 + \frac{g_1 \wedge
b_1}{g_0}.
 \end{split}
\end{gather}
\end{lemma}

For the calculations to come later, it will be fruitful to give yet another
version of the T-duality relationship explicitly in terms of the canonical
decomposition of
an $S^1$-invariant pair $(g, b)$ on a principal bundle which we now record.

\begin{lemma} \label{metricdecomp} A $S^1$-invariant metric on a principal
bundle with canonical vector field $e_{\theta}$ is uniquely
determined by a base metric, a family of fiber metrics, and a connection.  More
precisely, $g$ may be uniquely expressed 
\begin{align*}
g = \phi \theta \otimes \theta + h
\end{align*}
where
\begin{align*}
\phi =&\ g(\dth, \dth)\\
\theta =&\ \frac{g(\dth, \cdot)}{g(\dth,
\dth)}\\
h(\cdot, \cdot) =&\ g( \pi^{\theta} \cdot, \pi^{\theta} \cdot),
\end{align*}
and here $\pi^{\theta}$ is the horizontal projection determined by $\theta$,
i.e.
\begin{align*}
\pi^{\theta}(X) = X - \theta(X) \dth.
\end{align*}
\end{lemma}

\begin{lemma} \label{bflddecomp} Let $M$ denote the total space of an $S^1$
principal bundle.  Given $\theta$ a connection on $M$, an $S^1$ invariant
two-form $b$ admits a unique decomposition
\begin{align*}
b = \theta \wedge \eta + \mu
\end{align*}
where $\eta$ and $\mu$ are basic forms.
\begin{proof} Let $\eta = \dth \hook b$.  Obviously
$\eta(\dth) = 0$ and so $\eta$ is basic.  We may then
declare
\begin{align*}
\mu = b - \theta \wedge \eta
\end{align*}
Observe that
\begin{align*}
\dth \hook \mu  = \dth \hook b -
\dth \hook \left(\theta \wedge \eta \right) = \eta - \eta =
0,
\end{align*}
so that $\mu$ is basic as well.
\end{proof}
\end{lemma}

\begin{prop} \label{Tdualrules} Let $(M, H, \theta)$ and $(\bar{M}, \bar{H},
\bar{\theta})$ be topologically T-dual, and suppose $(g,b)$ is dual to
$(\bar{g}, \bar{b})$.  Let $\theta_g, \phi_g, h_g$ denote the
connection $1$-form, fiber metric, and base metric determined by $g$ via Lemma
\ref{metricdecomp}.  Furthermore, let $\eta_g$ and $\mu_g$ denote the basic
$1$-form and $2$-form associated to $b$ and $\theta_g$ via Lemma
\ref{bflddecomp}.  Then if $\theta_{\bar{g}}$, etc. denote the corresponding
data associated to $\bar{g}$, one has
\begin{align*}
\phi_{\bar{g}} =&\ \frac{1}{\phi_g}\\
\theta_{\bar{g}} =&\ \bar{\theta} + \eta_g\\
h_{\bar{g}} =&\ h_g\\
\eta_{\bar{g}} =&\ \theta_g - \theta\\
\mu_{\bar{g}} =&\ \mu_g - \eta_g \wedge \eta_{\bar{g}}.
\end{align*}
\begin{proof} First we compute
\begin{align*}
\theta_g = \theta + \frac{g_1}{g_0}.
\end{align*}
Then we obtain
\begin{align*}
\eta_g = \dth \hook b =&\ - b_1.
\end{align*}
Then we may express
\begin{align*}
\mu_g =&\ b - \theta_g \wedge \eta_g\\
=&\ b_1 \wedge \theta - \left( \theta + \frac{g_1}{g_0} \right) \wedge (-b_1) +
b_2\\
=&\ b_2 + \frac{g_1}{g_0} \wedge b_1.
\end{align*}
Furthermore we obtain
\begin{align*}
\bar{\theta}_{\bar{g}} = \bar{\theta} - b_1 = \bar{\theta} + \eta_g
\end{align*}
Then, according to the Buscher rules,
\begin{align*}
\eta_{\bar{g}} =&\ \bar{\dth} \hook b\\
=&\ \frac{g_1}{g_0}\\
=&\ \theta_g - \theta.
\end{align*}
Then we obtain
\begin{align*}
\mu_{\bar{g}} =&\ \bar{b} - \bar{\theta}_{\bar{g}} \wedge \eta_{\bar{g}}\\
=&\ - \frac{g_1}{g_0} \wedge \bar{\theta} + b_2 + \frac{g_1 \wedge b_1}{g_0} -
\left( \bar{\theta} + \eta_g \right) \wedge \left( \theta_g - \theta \right)\\
=&\ (\bar{\theta} - b_1) \wedge \frac{g_1}{g_0} + b_2 - \bar{\theta}_{\bar{g}}
\wedge (\frac{g_1}{g_0})\\
=&\ b_2\\
=&\ \mu_g - \frac{g_1}{g_0} \wedge b_1\\
=&\ \mu_g - \eta_{g} \wedge \eta_{\bar{g}}.
\end{align*}
\end{proof}
\end{prop}

\begin{lemma} \label{consistencylemma} Let $(M, H, \theta)$ and $(\bar{M},
\bar{H}, \bar{\theta})$ be topologically T-dual, and
 suppose $(g,b)$ is dual to $(\bar{g}, \bar{b})$.  Then (\ref{toptdual}) holds
for the quadruple $(H_b,\theta_g, \bar{H}_{\bar{b}}, \bar{\theta}_{\bar{g}})$.
\begin{proof} We directly compute (suppressing the presence of $p^*$ and
$\bar{p}^*$) using Proposition \ref{Tdualrules} that
\begin{align*}
H_{b} - \bar{H}_{\bar{b}}=&\ H + db - \bar{H} - d \bar{b}\\
=&\ H - \bar{H} + d \left( \theta_g \wedge \eta_g + \mu_g \right) - d \left(
\bar{\theta}_{\bar{g}} \wedge \eta_{\bar{g}} + \mu_{\bar{g}} \right)\\
=&\ d \left(\theta \wedge \bar{\theta} + \theta_g \wedge \eta_g -
\bar{\theta}_{\bar{g}} \wedge \eta_{\bar{g}} + \eta_g \wedge \eta_{\bar{g}}
\right)\\
=&\ d(\theta_g \wedge \bar{\theta}_{\bar{g}}).
\end{align*}
\end{proof}
\end{lemma}

\begin{lemma} Given $(g, b)$ and $(\bar{g}, \bar{b})$ T-dual data, if we declare
$\theta_g$ and $\bar{\theta}_g$ to be the background connections, which is valid
by Lemma \ref{consistencylemma}, then the pair $(g, 0)$ and $(\bar{g}, 0)$ is
$T$-dual with respect to this background.
\begin{proof} This follows immediately from Proposition \ref{Tdualrules}.
\end{proof}
\end{lemma}

\begin{lemma} \label{Hdecomp} If $\theta$ denotes a choice of connection, given
$H$ an $S^1$-invariant three-form, $H$ admits a unique decomposition
\begin{align*}
H = \theta \wedge Y + Z
\end{align*}
where $Y$ and $Z$ are basic forms.
\begin{proof} Following the proof of Lemma \ref{bflddecomp} we let $Y = \dth
\hook H$ and $Z = H - \theta \wedge Y$ and this is the required decomposition.
\end{proof}
\end{lemma}

Next we relate the three-form
decomposition of Lemma \ref{Hdecomp} for $T$-dual structures.

\begin{lemma} \label{Hrelation} Let $(M, g, b)$ and $(\bar{M}, \bar{g},
\bar{b})$ be T-dual data.  Then
\begin{align*}
Z =&\ \bar{Z}, \qquad Y = - \bar{F}_{\bar{\theta}}, \qquad \bar{Y} = -
F_{\theta}.
\end{align*}
\begin{proof} Let $\til{\dth}$ denote the vector field defining the action of
$S^1$ coming from the bundle $M$ induced on the fiber product $M \times_{S^1}
\bar{M}$.  Likewise define $\til{\bar{\dth}}$.  We compute
\begin{align*}
\pi^* Y =&\ \pi^* \left(\dth \hook H \right)\\
=&\ \til{\dth} \hook \pi^* H\\
=&\ \til{\dth} \hook \left(\bar{\pi}^* \bar{H} + d (\theta_g
\wedge \bar{\theta}_{\bar{g}}) \right)\\
=&\ \til{\dth} \hook \left( F_{\theta} \wedge
\bar{\theta}_{\bar{g}} - \theta_g \wedge \bar{F}_{\bar{\theta}} \right)\\
=&\ - \bar{F}_{\bar{\theta}}.
\end{align*}
The calculation of $\bar{\pi}^* \bar{Y}$ is identical.  Finally we have
\begin{align*}
\pi^* Z =&\ \pi^* \left( H - \theta \wedge Y \right)\\
=&\ \bar{\pi}^* \bar{H} + d \left( \theta_g \wedge \bar{\theta}_{\bar{g}}
\right) + \theta_g \wedge \bar{F}_{\bar{\theta}}.\\
=&\ \bar{\pi}^* \bar{H} + F_g \wedge \bar{\theta}_{\bar{g}}\\
=&\ \bar{\pi}^* \bar{H} + \bar{\theta}_g \wedge F_{\theta}\\
=&\ \bar{\pi}^* \bar{H} - \bar{\theta} \wedge \bar{Y}\\
=&\ \bar{\pi}^* \bar{Z}.
\end{align*}
\end{proof}
\end{lemma}

\section{Proof of Theorem \ref{Tdualflow}} \label{Tdualflowsec}

\begin{rmk} (\textbf{Notational Conventions}) Given an $S^1$-invariant metric as
in Lemma \ref{metricdecomp}, $\N$ will always denote the
covariant derivative with respect to the base metric $h$, whereas the covariant
derivative with respect to $g$ will be denoted $D$ and will be given bars when
necessary.
\end{rmk}

\subsection{Curvature Calculations}

In this subsection we record a number of curvature calculations necessary for
the proof of Theorem \ref{Tdualflow}.  To set up the calculations, we first
choose coordinates at
some point $p \in B$ corresponding to normal coordinates for $h$, the base
metric.  In such a local chart we can express the connection canonically as
$\theta = dy + A_i dx^i$.  Then over any point in $\pi^{-1} p$ we choose a local
frame field
\begin{align*}
e_i := \frac{\del}{\del x^i} - A_i \frac{\del}{\del y}.
\end{align*}
One directly obtains that $e_i \hook \theta = 0$ for all $i$.  Moreover, observe
that
\begin{align*}
[e_i,e_j] = \left( A_{i,j}^{\theta} - A_{j,i}^{\theta} \right)\frac{\del}{\del
y} = - F_{ij} \frac{\del}{\del y}, \qquad [e_i, \frac{\del}{\del y}] = 0,
\end{align*}
where $F$ denotes the curvature of $A$.  Also, note that
\begin{align*}
g(e_i, e_j) = h \left(\frac{\del}{\del x^i}, \frac{\del}{\del x^j} \right).
\end{align*}

\begin{lemma} \label{Riccicalc} With the setup above one has
\begin{gather} \label{KKcurv}
\begin{split}
R_{i j} =&\ ^h R_{ij} - \frac{\phi}{2} \FF_g - \frac{1}{2 \phi} \N_i \N_j
\phi + \frac{1}{4 \phi^2} \N_i \phi \N_j \phi,\\
R_{i\theta} =&\ \frac{\phi}{2} d^*_h F_i - \frac{3}{4} \left(\N \phi
\hook F\right)_i,\\
R_{\theta \theta} =&\ - \frac{1}{2} \gD \phi + \frac{1}{4 \phi} \brs{\N
\phi}^2 + \frac{\phi^2}{4} \brs{F}^2.
\end{split}
\end{gather}
\begin{proof} The proof is a straightforward calculation we include for
convenience.  First we compute the Christoffel symbols.  Note
\begin{gather} \label{KKChris}
\begin{split}
\gG_{ij}^k =&\ ^h \gG_{ij}^k, \qquad \gG_{ij}^{\theta} = - \frac{1}{2} F_{ij},
\qquad \gG_{i \theta}^k = \frac{\phi}{2} h^{kl} F_{il} = - \frac{\phi}{2} h^{kl}
F_{li},\\
\gG_{\theta \theta}^k =&\ - \frac{1}{2} \N^k \phi, \qquad \gG_{i
\theta}^{\theta} = \frac{1}{2
\phi} \N_i \phi, \qquad \gG_{\theta \theta}^{\theta} = 0.
\end{split}
\end{gather}
Observe also the general curvature formula
\begin{align*}
R_{\gb\gg} = R_{\ga\gb\gg}^{\ga} =&\ \del_{\ga} \gG_{\gb \gg}^{\ga} - \del_{\gb}
\gG_{\ga\gg}^{\ga} - \gG_{\ga \gg}^{\mu} \gG_{\gb \mu}^{\ga} + \gG_{\gb
\gg}^{\mu} \gG_{\ga \mu}^{\ga} - C_{\ga\gb}^{\mu} \gG_{\mu \gg}^\ga.
\end{align*}
Using these we may compute
\begin{align*}
 R_{ij} =&\ e_{\ga} \gG_{ i j}^{\ga} - e_i \gG_{\ga j}^{\ga} - \gG_{\ga j}^{\mu}
\gG_{i \mu}^{\ga} + \gG_{i j}^{\mu} \gG_{\ga \mu}^{\ga} - C_{\ga i}^{\mu}
\gG_{\mu j}^{\ga}\\
=&\ ^h R_{ij} - e_i \gG_{\theta j}^{\theta} - \gG_{\theta j}^{\theta} \gG_{i
\theta}^{\theta} - \gG_{k j}^{\theta} \gG_{i \theta}^k - \gG_{\theta j}^k \gG_{i
k}^{\theta}\\
&\ \quad + \gG_{i j}^{\theta} \gG_{k \theta}^k - C_{k i}^{\theta} \gG_{\theta
j}^k - C_{\theta i}^k \gG_{k j}^{\theta}\\
=&\ ^h R_{ij} - \frac{\phi}{2} \FF_g - \N_i \left( \frac{1}{2 \phi} \N_j \phi
\right) - \frac{1}{4 \phi^2} \N_i \phi \N_j \phi\\
=&\ ^h R_{ij} - \frac{\phi}{2} \FF_g - \frac{1}{2 \phi} \N_i \N_j \phi +
\frac{1}{4 \phi^2} \N_i \phi \N_j \phi.
\end{align*}
Next we have
\begin{align*}
R_{\theta i} =&\ e_{\ga} \gG_{\theta i}^{\ga} - e_{\theta} \gG_{\ga i}^{\ga} -
\gG_{\ga i}^{\mu} \gG_{\theta \mu}^{\ga} + \gG_{\theta i}^{\mu} \gG_{\ga
\mu}^{\ga} - C_{\ga \theta}^{\mu} \gG_{\mu i}^{\ga}\\
=&\ e_j \gG_{\theta i}^j - \gG_{\theta i}^{j} \gG_{\theta j}^{\theta} - \gG_{j
i}^{\theta} \gG_{\theta \theta}^j + \gG_{\theta i}^{\theta} \gG_{j \theta}^j +
\gG_{\theta i}^j \gG_{\theta j}^{\theta}\\
=&\ e_j \left( - \frac{\phi}{2} h^{jl} F_{li} \right) - \left( - \frac{\phi}{2}
h^{jl} F_{li} \right) \left( \frac{1}{2 \phi} \N_j \phi \right) - \left(-
\frac{1}{2} F_{ji} \right) \left( -\frac{1}{2} \N^j \phi \right)\\
&\ \qquad + \left( \frac{1}{2 \phi} \N_i \phi \right) \left( - \frac{\phi}{2}
h^{jl} F_{lj} \right) + \left( - \frac{\phi}{2} h^{jl} F_{li} \right) \left(
\frac{1}{2 \phi} \N_j \phi \right)\\
=&\ \frac{\phi}{2} d^*_h F_i - \frac{3}{4} \left(\N \phi \hook F\right)_i.
\end{align*}
Lastly we have
\begin{align*}
R_{\theta \theta} =&\ \del_i \gG_{\theta \theta}^i  - \gG_{\ga \theta}^{\mu}
\gG_{\mu \theta}^{\ga} + \gG_{\theta \theta}^{\mu} \gG_{\mu \ga}^{\ga}\\
=&\ - \frac{1}{2} \del_i \N^i \phi - \gG_{\theta \theta}^i \gG_{i
\theta}^{\theta} - \gG_{i \theta}^{\theta} \gG_{\theta \theta}^i - \gG_{i
\theta}^j \gG_{j \theta}^i + \gG_{\theta \theta}^i \gG_{i \theta}^{\theta}\\
=&\ - \frac{1}{2} \gD \phi + \frac{1}{4 \phi} \brs{\N \phi}^2 + \frac{\phi^2}{4}
\brs{F}_h^2.
\end{align*}
\end{proof}
\end{lemma}

\begin{lemma} \label{Hsqcalc} Let $(M^n, g, H)$ be $S^1$-invariant data.  Then
\begin{align*}
\HH_{ij} =&\  \frac{2}{\phi} \YY_{ij} + \ZZ_{ij},\\
\HH_{i\theta} =&\ \IP{ e_i \hook Z, Y },\\
\HH_{\theta \theta} =&\ \brs{Y}^2.
\end{align*}
\begin{proof} Using Lemma \ref{Hdecomp} we have
\begin{align*}
 \HH_{ij} =&\ g^{\ga \gb} g^{\gg\gd} H_{i \ga \gg} H_{j \gb \gd}\\
 =&\ \frac{1}{\phi} h^{kl} H_{i \theta k} H_{j \theta k} + \frac{1}{\phi} h^{kl}
H_{i k \theta} H_{j l \theta} + h^{kl}h^{mn} H_{ikm} H_{jln}\\
=&\ \frac{2}{\phi} h^{kl} Y_{ik} Y_{jl} + h^{kl} h^{mn} Z_{ikm} Z_{jln}\\
=&\ \frac{2}{\phi} \YY_{ij} + \ZZ_{ij}.
\end{align*}
Again by Lemma \ref{Hdecomp},
\begin{align*}
\HH_{i\theta} =&\ \IP{e_i \hook H, \dth \hook H}\\
=&\ \IP{ e_i \hook \left( \theta \wedge Y + Z \right), \dth \hook (\theta \wedge
Y + Z)}\\
=&\ \IP{ \theta \wedge \left( e_i \hook Y) + e_i \hook Z, Y \right}\\
=&\ \IP{ e_i \hook Z, Y },
\end{align*}
where the last line follows since $\theta$ is $g$-orthogonal to basic forms.
Lastly, using Lemma \ref{Hdecomp} we obtain
\begin{align*}
\HH_{\theta \theta} =&\ \IP{ \dth \hook H, \dth \hook H} = \brs{Y}^2.
\end{align*}
\end{proof}
\end{lemma}

\begin{lemma} \label{fldstrdecomp} Let $(M^n, g, H)$ be $S^1$-invariant data. 
Then
\begin{align*}
(d^*_g H)_{i \theta} =&\ d^*_h Y_i - \frac{\phi}{2} \IP{ e_i \hook Z, F} +
\frac{1}{2 \phi} \left( \N \phi \hook Y \right)_i\\
(d^*_g H)_{ij} =&\ \left(d^*_h Z \right)_{ij} - \frac{1}{2 \phi} \left( \N
\phi \hook Z \right)_{ij}.
\end{align*}
\begin{proof}
First
\begin{align*}
\left(d^*_g H \right)_{ij} =&\ - g^{\ga\gb} D_\ga H_{\gb ij}\\
=&\ - g^{\ga \gb} \left( \del_{\ga} H_{\gb i j} - \gG_{ \ga\gb}^{\mu} H_{\mu i
j}
- \gG_{\ga i}^{\mu} H_{\gb \mu j} - \gG_{\ga j}^{\mu} H_{\gb i \mu} \right)\\
=&\ \left(d^*_h Z \right)_{ij} + \frac{1}{\phi} \gG_{\theta \theta}^k Z_{k i j}
+ h^{kl} \gG_{k i}^{\theta} H_{l \theta j} + \frac{1}{\phi} \gG_{\theta i}^k
H_{\theta k j}\\
&\ + \frac{1}{\phi} \gG_{\theta j}^k H_{\theta i k} + h^{kl} \gG_{kj}^{\theta}
H_{l i \theta}\\
=&\ \left(d^*_h Z \right)_{ij} - \frac{1}{2 \phi} \left( \N \phi \hook Z
\right)_{ij}.
\end{align*}
Second
\begin{align*}
\left(d^*_g H \right)_{i\theta} =&\ - g^{\ga \gb} D_{\ga} (H_g)_{\gb i
\theta}\\
=&\ - g^{\ga \gb} \left( e_{\ga} H_{\gb i \theta} - \gG_{\ga \gb}^{\mu} H_{\mu i
\theta} - \gG_{\ga i}^{\mu} H_{\gb \mu \theta} - \gG_{\ga \theta}^{\mu} H_{\gb i
\mu} \right)\\
=&\ - h^{kl} D_k H_{l i \theta} + \frac{1}{\phi} \gG_{\theta \theta}^{j} H_{j i
\theta} + \frac{1}{\phi} \gG_{\theta \theta}^{k} H_{\theta i k} + h^{kl} \gG_{k
\theta}^m H_{l i m} + h^{kl} \gG_{k \theta}^{\theta} H_{l i \theta}\\
=&\ d^*_h Y_i + h^{kl} \left( - \frac{\phi}{2} h^{mn} F_{nk} \right) \left(
Z_{lim}
\right) + h^{kl} \left( \frac{1}{2 \phi} \N_k \phi \right) \left(Y_{li}
\right)\\
=&\ d^*_h Y_i - \frac{{\phi}}{2} \IP{ e_i \hook Z, F} + \frac{1}{2 \phi} \left(
\N
\phi \hook Y \right)_i.
\end{align*}
\end{proof}
\end{lemma}

\begin{lemma} \label{Hessiandecomp} Let $(M^n, g, H)$ be $S^1$-invariant data,
and let $\ga \in
T^*M$ be basic and $S^1$-invariant.  Then
\begin{align*}
D_i D_j f=&\ \N_i \N_j f, \qquad D_{\theta} D_i f = D_i D_{\theta} f =
\left(\frac{1}{2} \N f^{\sharp} \hook F\right)_i, \qquad
D_{\theta} D_{\theta} f = \frac{1}{2} \left< \N \phi, \N f \right>_h.
\end{align*}
\begin{proof} This follows directly from the the general calculation
\begin{align*}
 D_{I} \ga_{J} =&\ e_{I} \ga_{J} - \gG_{IJ}^{K} \ga_{K}.
\end{align*}
and the calculation of the Christoffel symbols in (\ref{KKChris}).
\end{proof}
\end{lemma}

\begin{lemma} \label{Liederdecomp} Let $(M^n, g, H)$ be $S^1$-invariant data,
and let $f \in
C^{\infty}(M)$ be $S^1$-invariant.  Then
\begin{align*}
\left( D f \hook H \right)_{i \theta} =&\ \left( \N f \hook Y \right)_i,\\
\left( D f \hook H \right)_{i j} =&\ \left( \N f \hook Z \right)_{ij}.
\end{align*}
\begin{proof} This follows immediately from Lemma \ref{Hdecomp}. 
\end{proof}
\end{lemma}

\subsection{Variational Calculations}

\begin{lemma} \label{variation} Let $(g_t, b_t)$ be a one-parameter family of
$S^1$-invariant data such that
\begin{align*}
\left. \dt g_t \right|_{t=0} =&\ k\\
\left. \dt b_t \right|_{|t=0} =&\ c.
\end{align*}
Let $\theta_t, h_t,$ etc. denote the unique data determining $g_t$ and $b_t$
determined by Lemmas \ref{metricdecomp} and \ref{bflddecomp} respectively.  Then
\begin{align*}
\left. \dt \phi_t  \right|_{t=0} =&\ k \left(\dth, \dth
\right) \\
\left. \dt \theta_t \right|_{t=0} =&\ \frac{k(\dth, \pi_H \cdot)}{\phi}\\
\left. \dt h_t \right|_{t=0} =&\ k(\pi_H \cdot, \pi_H \cdot)\\
\left. \dt \eta_t \right|_{t=0} =&\ \dth \hook c.\\
\left. \dt \mu_t \right|_{t=0} =&\ c -  \frac{k(\dth, \pi_H \cdot)}{\phi} \wedge
\eta- \theta \wedge \left( \dth \hook c \right)
\end{align*}
\begin{proof} First, using the formula $\phi_t = g_t(\dth,
\dth)$, differentiating immediately yields the first
equation.  Next, using Lemma \ref{metricdecomp} we differentiate and obtain
\begin{align*}
\left. \dt \theta_t \right|_{t=0} = \frac{k(\dth, \cdot)}{\phi} - \frac{g(\dth,
\cdot) k(\dth,\dth)}{\phi^2} = \frac{k(\dth, \pi_H \cdot)}{\phi}.
\end{align*}
Next we observe that $h_t = g_t(\pi_{H_t} \cdot, \pi_{H_t} \cdot)$, and so
\begin{align*}
\left. \dt h_t \right|_{t=0} =&\  k(\pi_{H} \cdot, \pi_H \cdot) - g \left(
\left(\dot{\theta} \cdot \right) \dth, \pi_H \cdot \right) - g\left(\pi_H \cdot,
\left(\dot{\theta} \cdot \right) \dth \right)\\
=&\ k(\pi_H \cdot, \pi_H \cdot),
\end{align*}
where the last line follows since $\dth$ is $g$-orthogonal to the image of
$\pi_H$.  Next, by definition, $\eta_t = \dth \hook b_t$. 
Differentiating this we immediately obtain
\begin{align*}
\left. \frac{\del}{\del t} \eta_t \right|_{t=0} = \dth \hook
\left(\left. \dt b_t \right|_{t=0}\right) = \dth \hook c.
\end{align*}
Lastly, we use the formula defining $\mu$ we obtain
\begin{align*}
\left. \dt \mu_t \right|_{t=0} =&\ \left. \dt \left( b_t - \theta_t \wedge
\eta_t \right) \right|_{t=0}\\
=&\ c -  \frac{k(\dth, \pi_H \cdot)}{\phi} \wedge \eta- \theta \wedge \left(
\dth \hook c \right).
\end{align*}
\end{proof}
\end{lemma}

\begin{lemma} \label{dualvariation} Let $(g_t, b_t)$ be a one-parameter family
of $S^1$-invariant data such that
\begin{align*}
\left. \dt g_t \right|_{t=0} =&\ k\\
\left. \dt b_t \right|_{|t=0} =&\ c.
\end{align*}
Let $\theta_t, h_t,$ etc. denote the unique data determining $g_t$ and $b_t$
determined by Lemmas \ref{metricdecomp} and \ref{bflddecomp} respectively, and
likewise define $\bar{\theta}_t$, etc.  Then
\begin{align*}
\left. \dt \bar{\phi}_t  \right|_{t=0} =&\ - \frac{k(\dth, \dth)}{\phi^2}\\
\left. \dt \bar{\theta}_t \right|_{t=0} =&\ \dth \hook c\\
\left. \dt \bar{h}_t \right|_{t=0} =&\ k(\pi_H \cdot, \pi_H \cdot)\\
\left. \dt \bar{\eta}_t \right|_{t=0} =&\ \frac{k(\dth, \pi_H \cdot)}{\phi}\\
\left. \dt \bar{\mu}_t \right|_{t=0} =&\ c - \left( \theta - \bar{\eta} \right)
\wedge \left( \dth \hook c \right).
\end{align*}
\begin{proof} We use the formulas of Proposition \ref{Tdualrules} and Lemma
\ref{variation} to conclude
\begin{align*}
\left. \dt \bar{\phi}_t \right|_{t=0} =&\ \left. \dt \frac{1}{\phi_t}
\right|_{t=0} = - \frac{1}{\phi^2} \left. \dt \phi_t \right|_{t=0} = -
\frac{k(\dth, \dth)}{\phi^2}.
\end{align*}
Next we compute
\begin{align*}
\left. \dt \bar{\theta}_t \right|_{t=0} =&\ \left. \dt \left( \bar{\theta} +
\eta_t \right) \right|_{t=0} = \left.\dt \eta_t \right|_{t=0} = \dth \hook c.
\end{align*}
Since $\bar{h}_t = h_t$ the third equation follows immediately.  For the fourth
we compute
\begin{align*}
\left. \dt \bar{\eta}_t \right|_{t=0}= \left. \dt \left(\theta_t - \theta\right)
\right|_{t=0} = \left. \dt \theta_t \right|_{t=0} = \frac{k(\dth, \pi_H
\cdot)}{\phi}.
\end{align*}
Lastly we compute
\begin{align*}
\left. \dt \bar{\mu}_t \right|_{t=0} =&\ \left. \dt \left( \mu_t - \eta_t \wedge
\bar{\eta}_{t} \right) \right|_{t=0}\\
=&\ c -  \frac{k(\dth, \pi_H \cdot)}{\phi} \wedge \eta- \theta \wedge \left(
\dth \hook c \right) - \left.\dt \left( \eta_t \wedge \bar{\eta}_t \right)
\right|_{t=0}\\
=&\ c - \frac{k(\dth, \pi_H \cdot)}{\phi} \wedge \eta- \theta \wedge \left( \dth
\hook c \right)\\
&\ - \left( \dth \hook c \right) \wedge \bar{\eta} - \eta \wedge \frac{k(\dth,
\pi_H \cdot)}{\phi}\\
=&\ c - \left( \theta - \bar{\eta} \right) \wedge \left( \dth \hook c \right).
\end{align*}
\end{proof}
\end{lemma}

\subsection{Preservation of T duality under the B-field flow}

\begin{proof}[Proof of Theorem \ref{Tdualflow}]

For the proof we will let $(g_t, b_t)$ be the solution to (\ref{RGflow}) with
initial condition $(g,b)$.  Now for all $t$ such that this flow exists smoothly,
let
$(\bar{g}_t, \bar{b}_t)$ denote the generalized metric which is dual to $(g_t,
b_t)$.  We aim to show that this one parameter family $(\bar{g}_t, \bar{b}_t)$
is a solution to (\ref{GFRGflow}) with $\bar{f}_t = f_t + \log \phi_t$, which
will
finish the proof.

To show that $(\bar{g}_t, \bar{b}_t)$ is the required solution to
(\ref{GFRGflow}) we will use
the decomposition of the data into $(\phi, \theta, g, \eta, \mu)$ given by
Lemmas \ref{metricdecomp} and \ref{bflddecomp} and compute the evolution of each
component.  In
particular let $(\phi_t, \theta_t, h_t, \eta_t, \mu_t)$ denote the decomposed
data associated to $(g_t, b_t)$, and likewise let $(\bar{\phi}_t,
\bar{\theta}_t, \bar{h}_t, \bar{\eta}_t, \bar{\mu}_t)$ denote the decomposed
data associated to the dualized metric $(\bar{g}_t, \bar{b}_t)$ according to
Proposition \ref{Tdualrules}.  We will
compute evolution equations for $\bar{\phi}_t$ etc. and show that this agrees
with the evolution induced by the B-field flow for the transformed data.  We
proceed with these five calculations.

\subsubsection{Evolution of $\bar{\phi}$}

First observe that by Lemma \ref{variation} and the curvature calculations of
Lemmas
\ref{Riccicalc} and \ref{Hsqcalc} and Lemma \ref{Hessiandecomp} we have that

\begin{align*}
\frac{\del \phi}{\del t} = \gD \phi - \frac{1}{2 \phi} \brs{\N \phi}^2 -
\frac{\phi^2}{2} \brs{F}^2 + \frac{1}{2} \brs{Y_H}^2 + \left< \N \phi, \N f
\right>
\end{align*}
By Lemma \ref{dualvariation} we thus obtain
\begin{align*}
\dt \bar{\phi} =&\ - \frac{1}{\phi^2} \left( \gD \phi - \frac{1}{2 \phi} \brs{\N
\phi}^2  -
\frac{\phi^2}{2} \brs{F}^2_h + \frac{1}{2} \brs{Y_H}^2 + \left< \N \phi, \N f
\right> \right)\\
=&\ \gD \left( \frac{1}{\phi} \right)  - 2 \frac{\brs{\N \phi}^2}{\phi^3}  +
\frac{1}{2 \phi^3} \brs{\N \phi}^2 + \frac{1}{2} \brs{F}^2_h - \frac{1}{2
\phi^2} \brs{Y_H}^2 - \frac{1}{\phi^2} \left<\N \phi, \N f \right>\\
=&\ \left(\gD \bar{\phi} - \frac{1}{2 \bar{\phi}} \brs{\N \bar{\phi}}^2 +
\frac{1}{2} \brs{F}^2_h - \frac{1}{2 \phi^2} \brs{Y_H}^2 \right) - \frac{\brs{\N
\phi}^2}{\phi^3} + \left< \N \bar{\phi}, \N f \right>
\end{align*}
Also we observe that, by Lemma \ref{Hrelation},
\begin{align*}
\brs{F}^2_h - \frac{1}{\phi^2} \brs{Y_H}^2 = \brs{\bar{Y}_{\bar{H}}}^2_h -
\bar{\phi}^2 \brs{\bar{F}}^2_h.
\end{align*}
Also, by Lemma \ref{Hessiandecomp} we have
\begin{align*}
2 \bar{D}_{\theta} \bar{D}_{\theta} \log \phi =&\ \left< \N \bar{\phi}, \N \log
\phi \right>_{\bar{h}} = - \frac{\brs{\N \phi}_h^2}{\phi^3}.
\end{align*}
Combining these calculations and again comparing against Lemmas \ref{Riccicalc}
and \ref{Hsqcalc} we obtain
\begin{align*}
 \dt \bar{\phi} =&\ \left(-2 \bar{\Rc} + \frac{1}{2} \bar{\HH} + 2 \bar{D}^2
(f + \log \phi) \right)_{\theta \theta} = \left(-2 \bar{\Rc} + \frac{1}{2}
\bar{\HH} + L_{\bar{D} (f + \log \phi)} \bar{g}\right)_{\theta \theta}
\end{align*}
as required.
\subsubsection{Evolution of $\bar{\theta}$}

First observe that by Lemma \ref{dualvariation} and \ref{fldstrdecomp} we have
that
\begin{align*}
\left( \dt \bar{\theta} \right)_i =&\ \left(\dth \hook (-
d^*_g
H_g + i_{D f} \hook H ) \right)_i\\
=&\ \left( d^*_g H \right)_{i \theta} - (\N f \hook Y)_i\\
=&\  d^*_h Y_i - \frac{\phi}{2} \IP{e_i \hook Z, F} +
\frac{1}{2 \phi} (\N \phi \hook Y)_i - (\N f \hook Y)_i\\
=&\ - d^*_h \bar{F}_i - \frac{\phi}{2} \IP{e_i \hook Z, F} - \frac{1}{2 \phi}
(\N \phi \hook \bar{F})_i + (\N f \hook \bar{F})_i
\end{align*}
But on the other hand by Lemmas \ref{Riccicalc}, \ref{Hsqcalc} and
\ref{Hessiandecomp}
\begin{align*}
\frac{-2}{\bar{\phi}} \left( \bar{\Rc} - \frac{1}{4} \bar{\HH} - \frac{1}{2}
L_{\bar{D} (f + \log \phi)} \bar{g} \right)_{i\theta} =&\ \frac{-2}{\bar{\phi}}
\left( \bar{\Rc} - \frac{1}{4} \bar{\HH} - \bar{D}^2 (f + \log
\phi) \right)_{i\theta}\\
=&\ - d^*_h \bar{F}_i +
\frac{3}{2 \bar{\phi}}
\left(\N \bar{\phi} \hook \bar{F} \right)_i + \frac{1}{2 \bar{\phi}} \left< e_i
\hook Z_{\bar{H}}, Y_{\bar{H}} \right> + \left(\N (f + \log \phi) \hook \bar{F}
\right)_i\\
=&\ - d^*_h \bar{F}_i - \frac{\phi}{2} \left< e_i \hook Z_H, F \right> -
\frac{1}{2 \phi} \left( \N \phi \hook \bar{F} \right)_i + (\N f \hook
\bar{F})_i.
\end{align*}
Combining these two calculations and using Lemma \ref{variation} again yields
the result.

\subsubsection{Evolution of $\bar{h}$}

First observe that by Lemma \ref{dualvariation} and the curvature calculations
of Lemmas
\ref{Riccicalc} and \ref{Hsqcalc} we have that
\begin{align*}
\dt \bar{h}_{ij} =  \dt h_{ij} =&\ - 2 \left( ^h \Rc - \frac{\phi}{2} \FF -
\frac{1}{2
\phi} \N \N \phi +
\frac{1}{4 \phi^2} \N \phi \otimes \N \phi - \frac{1}{4} \HH - \frac{1}{2} L_{D
f} g \right)_{ij}
\end{align*}
First observe that since $h = \bar{h}$, we have $^h \Rc =\ ^{\bar{h}} \Rc$. 
Next, using Lemma \ref{Hsqcalc} we observe that
\begin{align*}
\left( - \frac{\phi}{2} \FF - \frac{1}{4} \HH \right)_{ij} =&\ \left(-
\frac{\phi}{2} \FF -
\frac{1}{4} \left( \frac{2}{\phi} \bar{\FF} + \ZZ \right) \right)_{ij}\\
=&\ \left( - \frac{\bar{\phi}}{2} \bar{\FF} - \frac{1}{4} \left(
\frac{2}{\bar{\phi}}
\FF + \bar{\ZZ} \right) \right)_{ij}\\
=&\ \left(- \frac{\bar{\phi}}{2} \bar{\FF} - \frac{1}{4} \bar{\HH} \right)_{ij}.
\end{align*}
Furthermore, a direct calculation using Lemma \ref{Hessiandecomp} yields
\begin{align*}
- \frac{1}{2 \phi} \N_i \N_j \phi + \frac{1}{4 \phi^2} \N_i \phi \N_j \phi =&\ -
\frac{\bar{\phi}}{2} \N_i \N_j \left( \frac{1}{\bar{\phi}} \right) + \frac{1}{4
\bar{\phi}^2} {\N}_i \bar{\phi} {\N}_j \bar{\phi}\\
=&\ \frac{\bar{\phi}}{2} \N_i \left( \frac{\N_j \bar{\phi}}{\bar{\phi}^2}
\right) + \frac{1}{4 \bar{\phi}^2} {\N}_i \bar{\phi} {\N}_j \bar{\phi}\\
=&\ \frac{1}{2 \bar{\phi}} \N_i \N_j \bar{\phi} - \frac{1}{\bar{\phi}^2} \N_i
\bar{\phi} \N_j \bar{\phi} + \frac{1}{4 \bar{\phi}^2} {\N}_i \bar{\phi} {\N}_j
\bar{\phi}\\
=&\ - \frac{1}{2 \bar{\phi}} \N_i \N_j \bar{\phi} + \frac{1}{4 \bar{\phi}^2}
\N_i \bar{\phi} \N_j \bar{\phi} + \left( \frac{1}{\bar{\phi}} \N_i \N_j
\bar{\phi} - \frac{1}{\bar{\phi}^2} \N_i \bar{\phi} \N_j \bar{\phi} \right)\\
=&\ - \frac{1}{2 \bar{\phi}} \N_i \N_j \bar{\phi} + \frac{1}{4 \bar{\phi}^2}
\N_i \bar{\phi} \N_j \bar{\phi} + \N_i \N_j \log \bar{\phi}\\
=&\ - \frac{1}{2 \bar{\phi}} \N_i \N_j \bar{\phi} + \frac{1}{4 \bar{\phi}^2}
\N_i \bar{\phi} \N_j \bar{\phi} - \bar{D}_i \bar{D}_j \log \phi.
\end{align*}
Combining these calculations and using Lemma \ref{Hessiandecomp} again yields
\begin{align*}
\dt \bar{h}_{ij} =&\ -2 \left( \ ^{\bar{h}} {\Rc}_{ij} - \frac{1}{2}
\bar{\FF}_{ij} - \frac{1}{2
\bar{\phi}}
{\N}_i {\N}_j \bar{\phi} + \frac{1}{4 \bar{\phi}^2} {\N}_i
{\phi} {\N}_j \bar{\phi}
- \frac{1}{4} \bar{\HH}_{ij} - \bar{D}_i \bar{D}_j \log \phi - \frac{1}{2} L_{\N
f} g \right)\\
=&\ -2 \left( \bar{\Rc} - \frac{1}{4} \bar{\HH} - \frac{1}{2} L_{\bar{D} (f +
\log \phi)} \bar{g}
\right)_{ij}
\end{align*}
as required.

\subsubsection{Evolution of $\eta$}

First observe that by Lemma \ref{variation}, the curvature calculations of
Lemmas
\ref{Riccicalc} and \ref{Hsqcalc}, Lemma \ref{Hessiandecomp}, and Lemma
\ref{Hrelation} we have that
\begin{align*}
\left(\dt \bar{\eta}\right)_i =  \left(\dt \theta\right)_i =&\ - d^*_h F_i +
\frac{3}{2 \phi} \left( \N \phi \hook F \right)_i - \frac{1}{2 \phi} \IP{e_i
\hook Z, \bar{F}} + (\N f \hook F)_i.
\end{align*}
On the other hand by Lemma \ref{fldstrdecomp}, and \ref{Hrelation} we have
\begin{align*}
\left( \bar{\dth} \hook - d^*_{\bar{g}} \bar{H} \right)_{i}
=&\ d^*_{\bar{h}} \bar{Y}_i + \frac{1}{2 \bar{\phi}} \left( \N \bar{\phi} \hook
\bar{Y}
\right)_i - \frac{\bar{\phi}}{2} \left< e_i \hook \bar{Z},
\bar{F} \right>\\
=&\ - d^*_h F_i + \frac{1}{2 \phi} \left( \N \phi \hook F \right)_i - \frac{1}{2
\phi} \IP{e_i \hook Z, \bar{F}}.
\end{align*}
Also, by Lemma \ref{Liederdecomp} and \ref{Hrelation} we have
\begin{align*}
\left( \bar{\dth} \hook i_{\bar{D} (f + \log \phi)} \bar{H}
\right)_i =&\ (\N f \hook F)_i + \frac{1}{\phi} \left( \N \phi \hook F
\right)_i.
\end{align*}
Collecting these calculations gives the required equality.

\subsubsection{Evolution of $\mu$}

Directly calculating using Lemmas \ref{fldstrdecomp}, \ref{variation}, and
\ref{dualvariation}
 we compute
\begin{align*}
\left( \dt \bar{\mu} \right)_{ij} =&\ \left( \dt \mu \right)_{ij}\\
=&\ \left( - d^*_g H + i_{\N f} \hook H \right)_{ij}\\
=&\ \left( - d^*_h Z  + \frac{1}{2 \phi} \left(\N \phi \hook Z \right) +
i_{\bar{D} f} \hook H \right)_{ij}\\
=&\ \left( - d^*_h \bar{Z} -\frac{1}{2 \bar{\phi}} \left( \N \bar{\phi} \hook
\bar{Z} \right) + i_{\bar{D} f} \hook H \right)_{ij}\\
=&\ \left( - d^*_{\bar{g}} \bar{H} - \frac{1}{\bar{\phi}} \left( \N \bar{\phi}
\hook \bar{Z} \right) + i_{\bar{D} f} \hook H \right)_{ij}\\
=&\ \left( - d^*_{\bar{g}} \bar{H} + \frac{1}{\phi} \left( \N \phi \hook \bar{Z}
\right) + i_{\bar{D} f} \hook H \right)_{ij}\\
=&\ \left( - d^*_{\bar{g}} \bar{H} + i_{\bar{D} (f + \log \phi)} \hook \bar{H}
\right)_{ij}
\end{align*}
\end{proof}

\subsection{Examples}

\begin{ex} We begin with a simple example to illustrate how
T-duality affects solutions to (\ref{RGflow}).  Let $M \cong S^3$ and consider
the Hopf fibration $S^1 \to S^3 \to S^2$, and let $\theta$ denote the connection
one form on $S^3$ satisfying 
$d \theta = \gw_{S^2}$, where $\gw_{S^2}$ denotes the standard area form on
$S^2$, and furthermore let $H = 0$.  Next let $\bar{M} \cong S^2 \times S^1$,
and consider the trivial
fibration $S^1 \to S^1 \times S^2 \to S^2$.  Let $\bar{\theta}$ denote the
pullback of the canonical line element on $S^1$ to $\bar{M}$, and let $\bar{H} =
- \bar{\theta} \wedge \gw_{S^2}$.  Certainly $d \bar{H} = 0$.  
Moreover, with the notation of \S \ref{Tdualtop}, observe that
\begin{align*}
p^* H - \bar{p}^* \bar{H} = \bar{p}^* \left( \gw_{S^2} \wedge \bar{\theta}
\right) = d p^*\theta \wedge \bar{\theta} = d \left( p^* \theta \wedge \bar{p}^*
\bar{\theta} \right).
\end{align*}
Thus $(M, H, \theta)$ and $(\bar{M}, \bar{H}, \bar{\theta})$ are topologically
T-dual.  Let $g_{S^2}$ denote the round metric on $S^2$ and consider an
$S^1$-invariant metric of the form
\begin{align*}
g = A \theta \otimes \theta + B g_{S^2}.
\end{align*}
Observe that by applying Proposition \ref{Tdualrules} we obtain that $(g, 0)$ is
T-dual to $(\bar{g}, \bar{b})$ with
\begin{gather} \label{Hopfdual}
\begin{split}
\bar{g} =&\ \frac{1}{A} \bar{\theta} \otimes \bar{\theta} + B g_{S^2},\\
\bar{b} =&\ 0.
\end{split}
\end{gather}
The solution to (\ref{RGflow}) with initial condition $(g, 0)$ on $M$ is given
by the Ricci flow, which takes the form
\begin{align*}
\dot{A} =&\ - \frac{A^2}{B^2}, \qquad \dot{B} = - 2 + \frac{A}{B}.
\end{align*}
Expressing the T-dual data as $\bar{g} = \bar{A} \bar{\theta} \otimes
\bar{\theta} + \bar{B} g_{S^2}$ and using (\ref{Hopfdual}) we obtain the
evolution equation for $\bar{g}$ as
\begin{align*}
\dot{\bar{A}} =&\ \frac{1}{B^2}, \qquad \dot{\bar{B}} = -2 + \frac{A}{B},
\end{align*}
which, comparing against Lemmas \ref{Riccicalc} and \ref{Hsqcalc}, is the
solution to (\ref{RGflow}).  Observe that $M$ shrinks to a round point under the
flow, whereas on $\bar{M}$ the $S^2$ shrinks to a point while the $S^1$ fiber
blows up.
\end{ex}

\begin{ex} More generally, we may let $M \cong S^{2n+1}$ and consider the Hopf
fibration 
$S^1 \to S^{2n+1} \to \mathbb C \mathbb P^n$, and let $\theta$ denote the
connection
one form on $S^{2n+1}$ satisfying $d \theta = \gw_{FS}$, where $\gw_{FS}$ is the
K\"ahler
form of the Fubini-Study metric on $\mathbb C \mathbb P^n$, and furthermore let
$H = 0$.  Next let $\bar{M} \cong \mathbb C \mathbb P^n \times S^1$, and
consider the trivial
fibration $S^1 \to S^1 \times \mathbb C \mathbb P^n \to \mathbb C \mathbb P^n$. 
Let $\bar{\theta}$ denote the pullback of the canonical line element on $S^1$ to
$\bar{M}$, and let $\bar{H} = - \bar{\theta} \wedge \gw_{S^2}$.  As in the
previous example one easily checks that $(M, H, \theta)$ and $(\bar{M}, \bar{H},
\bar{\theta})$ are topologically T-dual.

Now let $g_0$ denote any metric on $S^{2n+1}$ with positive curvature operator.
Consider the solution to (\ref{RGflow}) with initial condition $(g_0, 0)$.  One 
observes that by the maximum principle the condition $H_0 \equiv 0$ is
preserved 
by (\ref{RGflow}), and so the solution $(g_t, b_t) = (g_t, 0)$, where $g_t$ is
the
 unique solution to Ricci flow with initial condition $g_0$.  By the theorem of
Bohm-Wilking \cite{BW}, we have that $g_t$ exists on some finite time interval
$[0, T)$, and converges to a round point as $t \to T$.  It follows from
Proposition \ref{Tdualrules} that the dual solution $(\bar{g}_t, \bar{b}_t)$
also exists on a finite time interval, asymptotically converging to a solution
which
homothetically shrinks the $\mathbb C \mathbb P^2$ base and expands the $S^1$
fiber, analogously to the previous example.
\end{ex}

\begin{remark} Any of the Ricci flow ``sphere theorems'' in odd dimension (for
instance
\cite{BS}, \cite{Ham3}) can be used in the above example to
generate long time existence results for (\ref{RGflow}).
\end{remark}

\bibliographystyle{hamsplain}

\end{document}